\documentclass[11pt,reqno]{amsart}

\usepackage{graphicx}

\usepackage{graphicx}
\usepackage[table]{xcolor}
\usepackage{xspace,tikz}
\usepackage{amsmath,amsthm,amsfonts,amssymb,latexsym,mathrsfs,color,extarrows}
\usepackage{hyperref}

\usepackage{pgfplots}
\pgfplotsset{compat=newest}
\usetikzlibrary{shapes.geometric,arrows,fit,matrix,positioning}
\tikzset
{
    treenode/.style = {circle, draw=black, align=center, minimum size=1cm},
    subtree/.style  = {isosceles triangle, draw=black, align=center, minimum height=0.5cm, minimum width=1cm, shape border rotate=90, anchor=north}
}

\newtheorem{theorem}{Theorem}
\newtheorem{corollary}[theorem]{Corollary}
\newtheorem{proposition}[theorem]{Proposition}

\newtheorem{lemma}[theorem]{Lemma}

\newtheorem{problem}[theorem]{Problem}

\newcommand{\lrf}[1]{\lfloor #1\rfloor}

\newcommand{\sgn}{{\rm sgn\,}}

\DeclareMathOperator{\N}{\mathbb{N}}

\title{A symmetric decomposition of the Boros-Moll polynomials}
\author{Guo-Niu Han}
\address{I.R.M.A., UMR 7501, Universit\'e de Strasbourg et CNRS, 7 rue Ren\'e Descartes, F-67084 Strasbourg, France}
\email{guoniu.han@unistra.fr (G.-N.~Han)}
\author[S.-M.~Ma]{Shi-Mei Ma}
\address{School of Mathematics and Statistics,
        Northeastern University at Qinhuangdao,
         Hebei 066000, P.R. China}
\email{shimeimapapers@163.com (S.-M. Ma)}
\author[Y.-N. Yeh]{Yeong-Nan Yeh}
\address{College of Mathematics and Physics, Wenzhou University, Wenzhou 325035, P.R. China}
\email{mayeh@math.sinica.edu.tw (Y.-N. Yeh)}
\begin{document}
\maketitle
\begin{abstract}
In their study of a quartic integral, Boros and Moll introduced a special
case of Jacobi polynomials, which are now known as the Boros-Moll polynomials.
In this paper, we study a symmetric decomposition of Boros-Moll polynomials. We discover that
both of the polynomials in the symmetric decomposition are alternatingly gamma-positive polynomials.
\bigskip

\end{abstract}
\date{\today}


\section{Introduction}\label{Section01}
Boros-Moll~\cite{Boros9901} discovered that for any $x\geqslant -1$ and $m\in\N$,
$$\int_0^{\infty}\frac{1}{(1+2xy^2+y^4)^{m+1}}\mathrm{d}y=\frac{\pi}{2^{m+3/2}(x+1)^{m+1/2}}P_m(x),$$
where $P_m(x)$ is now known as the {\it Boros-Moll polynomial}.
Using Ramanujan's Master Theorem, Boros-Moll~\cite{Boros2001} found that the coefficients of $P_m(x)$ are positive. In fact, they
showed that
$$P_m(x)=2^{-2m}\sum_k 2^k\binom{2m-2k}{m-k}\binom{m+k}{k}(x+1)^k,$$
which indicates that
$P_m(x)=\sum_{i=0}^md_i(m)x^i$,
where
$$d_i(m)=2^{-2m}\sum_{k=i}^m2^k\binom{2m-2k}{m-k}\binom{m+k}{k}\binom{k}{i}.$$
In~\cite{Boros9902}, they showed that the sequence $P_m(x)$ is unimodal and the model of it appears in the middle.
For example, $$P_5(x)=\frac{4389}{256}+\frac{8589}{128}x+\frac{7161}{64}x^2+\frac{777}{8}x^3+\frac{693}{16}x^4+\frac{63}{8}x^5.$$
By using the RISC package MultiSum, Kauers-Paule~\cite[Eq.~(6)]{Kausers07} found that for $0\leqslant i\leqslant m+1$,
the numbers $d_i(m)$ satisfy the recurrence relation
\begin{equation}\label{recu01}
2(m+1)d_i(m+1)=2(m+i)d_{i-1}(m)+(4m+2i+3)d_i(m).
\end{equation}

Much progress has been made since Boros and Moll proved the positivity of $d_i(m)$, see~\cite{Boros9902,Boros2001,Chen09,Chen10} and references therein.
In particular, based on the structure of reluctant functions introduced by
Mullin and Rota along with an extension of Foata's bijection between
Meixner endofunctions and bi-colored permutations, Chen-Pang-Qu~\cite{Chen10} found a combinatorial proof of the positivity of $d_i(m)$.
The purpose of this paper is to show that $P_m(x)$ is alternatingly bi-$\gamma$-positive.

The classical Jacobi polynomials $P_m^{(\alpha,\beta)}(x)$ can be defined by
$$P_m^{(\alpha,\beta)}(x)=\sum_{k=0}^m(-1)^{m-k}\binom{m+\beta}{m-k}\binom{m+k+\alpha+\beta}{k}\left(\frac{1+x}{2}\right)^k.$$
According to~\cite{Boros9902}, the Boros-Moll polynomial $P_m(x)$
can be viewed as the Jacobi polynomial $P_m^{(\alpha,\beta)}(x)$ with $\alpha=m+\frac{1}{2}$ and $\beta=-\left(m+\frac{1}{2}\right)$.
In the same way as discussed in Section~\ref{Section02}, one may give an answer to the following problem and we leave it for further research.
\begin{problem}
Under what conditions the polynomials $P_m^{(\alpha,\beta)}(x)$ are alternatingly bi-$\gamma$-positive?
\end{problem}

In the rest of this section, we collect some definitions.
Let $f(x)=\sum_{i=0}^nf_ix^i$ be a polynomial with nonnegative coefficients.
If $f(x)$ is symmetric with the centre of symmetry $\lrf{n/2}$, i.e., $f_i=f_{n-i}$ for all indices $0\leqslant i\leqslant n$,
then it can be expanded as
$$f(x)=\sum_{k=0}^{\lrf{{n}/{2}}}\gamma_kx^k(1+x)^{n-2k}.$$
Following Gal~\cite{Gal05}, the sequence $\{\gamma_k\}_{k=0}^{\lrf{n/2}}$ is called the {\it $\gamma$-vector} of $f(x)$.
If $\gamma_k\geqslant 0$ for all $0\leqslant k\leqslant \lrf{n/2}$, then $f(x)$ is said to be {\it $\gamma$-positive}.
Clearly, $\gamma$-positivity implies symmetry and unimodality.
We say that $f(x)$ is {\it alternatingly $\gamma$-positive} if its $\gamma$-vector alternates in sign.
For example, $1+x+x^2$ is alternatingly $\gamma$-positive, since $1+x+x^2=(1+x)^2-x$.
We now recall an elementary result.
\begin{proposition}[{\cite{Beck2010,Branden18}}]\label{prop01}
Let $f(x)$ be a polynomial of degree $n$.
There is a unique symmetric decomposition $f(x)= a(x)+xb(x)$, where
\begin{equation}\label{ax-bx-prop01}
a(x)=\frac{f(x)-x^{n+1}f(1/x)}{1-x},~b(x)=\frac{x^nf(1/x)-f(x)}{1-x}.
\end{equation}
When $f(0)\neq0$, we have $\deg a(x)=n$ and $\deg b(x)\leqslant n-1$.
\end{proposition}
We call the ordered pair of polynomials $(a(x),b(x))$ the {\it symmetric decomposition} of $f(x)$, since $a(x)$ and $b(x)$ are both symmetric polynomials.
Let $(a(x),b(x))$ be the symmetric decomposition of $f(x)$.
We say that $f(x)$ is {\it alternatingly bi-$\gamma$-positive} if both $a(x)$ and $b(x)$ are alternatingly $\gamma$-positive.
Very recently, the symmetric decompositions of various polynomials have been studied in geometric and enumerative combinatorics, see~\cite{Athanasiadis20,Branden18} and references therein.

The polynomial $f(x)$ is said to be {\it unimodal} if
$f_0\leqslant f_1\leqslant \cdots\leqslant f_k\geqslant f_{k+1}\geqslant\cdots \geqslant f_n$ for some $k$, where the index $k$ is called the {\it mode} of $f(x)$.
The polynomial $f(x)$ is {\it spiral} if
$$f_n\leqslant f_0\leqslant f_{n-1}\leqslant f_1\leqslant \cdots\leqslant f_{\lrf{n/2}}.$$
Following~\cite[Definition 2.9]{Schepers13}, the polynomial $f(x)$ is {\it alternatingly increasing} if
$$f_0\leqslant f_n\leqslant f_1\leqslant f_{n-1}\leqslant\cdots \leqslant f_{\lrf{{(n+1)}/{2}}}.$$
If $f(x)$ is spiral and $\deg f(x)=n$, then $x^nf(1/x)$ is alternatingly increasing, and vice versa.
\begin{lemma}[{\cite[Lemma~2.1]{Beck2019}}]\label{lemma-alt}
Let $(a(x),b(x))$ be the symmetric decomposition of $f(x)$, where $\deg f(x)=\deg a(x)=n$ and $\deg b(x)=n-1$.
Then $f(x)$ is alternatingly increasing if and only if both $a(x)$ and $b(x)$ are unimodal.
\end{lemma}
In~\cite[Corollary~1.2]{Chen09}, Chen-Xia found that $P_m(x)$ is spiral, then $x^mP_m\left(\frac{1}{x}\right)$ is alternatingly increasing.
By Lemma~\ref{lemma-alt}, one can get the following result.
\begin{proposition}
For any $m\geqslant 1$, let $(p_m(x),q_m(x))$ be the symmetric decomposition of $x^mP_m\left(\frac{1}{x}\right)$.
Then $p_m(x)$ and $q_m(x)$ are both unimodal.
\end{proposition}
\section{Main results}\label{Section01}
Define
$$Q_m(x)=2^mm!x^mP_m\left(\frac{1}{x}\right)=\sum_{i=0}^mc_i(m)x^i.$$
Note that
$c_i(m)=2^mm!d_{m-i}(m)$. It follows from~\eqref{recu01} that
\begin{equation}\label{cim}
c_i(m+1)=(4m-2i+2)c_i(m)+(6m-2i+5)c_{i-1}(m),
\end{equation}
with the initial conditions $c_0(0)=1$ and $c_i(0)=0$ for all $i\neq 0$.
By~\eqref{cim}, we obtain
\begin{equation}\label{Qmx-recu}
Q_{m+1}(x)=(2m+1)(2+3x)Q_m(x)-2x(1+x)\frac{\mathrm{d}}{\mathrm{d}x}Q_m(x),~Q_0(x)=1.
\end{equation}
Clearly, $Q_m(-1)=(-1)^m(2m-1)!!$ and $Q_m(0)=(2m)!/m!$.
It should be noted that $Q_m(0)$ counts binary rooted plane trees, with $n$ labeled end nodes of degree $1$, see~\cite[A001813]{Sloane}.
Following~\cite[A334907]{Sloane}, one has
\begin{equation*}
Q_m(1)=2^mm!P_m(1)=\frac{m!}{2^{m+1}}\binom{4m+2}{2m+1}.
\end{equation*}

Below are the symmetric decomposition of the polynomials $Q_m(x)$ for $m\leqslant 4$:
\begin{align*}
Q_1(x)&=2(1+x)+x,~Q_2(x)=3(4+7x+4x^2)+9x(1+x),\\
Q_3(x)&=3(40+103x+103x^2+40 x^3)+3x(37+69x+37 x^2),\\
Q_4(x)&=105(16+55x+79x^2+55x^3+16 x^4)+255x(1+x)(7+12x+7x^2),\\
Q_5(x)
&=315(96+415x+781x^2+781x^3+415x^4+96x^5)+\\
&\quad 315x(113+403x+583x^2+403x^3+113x^4).
\end{align*}

\begin{theorem}\label{thm02}
The Boros-Moll polynomials are alternatingly bi-$\gamma$-positive. More precisely,
let $(a_m(x),b_m(x))$ be the symmetric decomposition of $Q_m(x)$.
Then
\begin{align*}
a_m(x) &=\sum_{k=0}^{\lrf{m/2}} \alpha_{m,k} x^k (1+x)^{m-2k},~
b_m(x) =\sum_{k=0}^{\lrf{(m-1)/2}} \beta_{m,k}  x^k (1+x)^{m-1-2k},
\end{align*}
where both $a_m(x)$ and $b_m(x)$ are alternatingly $\gamma$-positive.
Thus $P_m(x)$ has the expansion:
$$P_m(x)=\frac{1}{2^mm!}\left(\sum_{k=0}^{\lrf{m/2}} \alpha_{m,k} x^k (1+x)^{m-2k}+\sum_{k=0}^{\lrf{(m-1)/2}} \beta_{m,k}  x^k (1+x)^{m-1-2k}\right).$$
\end{theorem}

Define
\begin{align*}
\alpha_m(x) &=\sum_{k=0}^{\lrf{m/2}} (-1)^k\alpha_{m,k} x^k,~
\beta_m(x) =\sum_{k=0}^{\lrf{(m-1)/2}} (-1)^k\beta_{m,k} x^k.
\end{align*}
Then
\begin{align*}
a_m(x) &=(1+x)^m\alpha_m\left(\frac{-x}{(1+x)^2}\right),~b_m(x)=(1+x)^{m-1}\beta_m\left(\frac{-x}{(1+x)^2}\right).
\end{align*}
By Lemma~\ref{lemma:alphabeta}, it is routine to deduce the following result.
\begin{corollary}
The polynomials $\alpha_m(x)$ and $\beta_m(x)$ satisfy the recurrence system:
\begin{equation}\label{LnELnO}
\left\{
  \begin{array}{l}
\alpha_{m+1}(x)=(4m+2)\alpha_m(x)-2x\frac{\mathrm{d}}{\mathrm{d}x}\alpha_m(x)+3x\beta_m(x)+4x^2\frac{\mathrm{d}}{\mathrm{d}x}\beta_m(x), \\
\beta_{m+1}(x)=\alpha_m(x)+4x\frac{\mathrm{d}}{\mathrm{d}x}\alpha_m(x)+(4m+3)\beta_m(x)+2x\frac{\mathrm{d}}{\mathrm{d}x}\beta_m(x), \\
  \end{array}
\right.
\end{equation}
with the initial conditions $\alpha_0(x)=1$ and $\beta_0(x)=0$.
\end{corollary}

If the coefficients of $\alpha_m(x)$ are positive, then so does $(4m+2)\alpha_m(x)-2x\frac{\mathrm{d}}{\mathrm{d}x}\alpha_m(x)$.
Using~\eqref{LnELnO} and by induction on $m$, we see that both $\alpha_m(x)$ and $\beta_m(x)$ have only positive coefficients.
For convenience, we list these polynomials for $m\leqslant 5$:
\begin{align*}
\alpha_1(x)&=2,~\beta_1(x)=1,~
\alpha_2(x)=12+3x,~\beta_2(x)=9,~
\alpha_3(x)=120+51x,\\
\beta_3(x)&=111+15x,~
\alpha_4(x)=1680+945x+105x^2,~\beta_4(x)=1785 + 510 x,\\
\alpha_5(x)&=30240+20475x+5040x^2,~\beta_5(x)=35595 + 15435 x + 945 x^2.
\end{align*}
In particular, $\alpha_{m,0}=(2m)!/m!$ and the numbers $\beta_{m,0}$ appear as the second column of the triangle~\cite[A000369]{Sloane}.
\section{Proof of Theorem~\ref{thm02}}\label{Section02}
Applying the formula~\eqref{ax-bx-prop01}, we see that $Q_m(x)=a_m(x)+xb_m(x)$, where
\begin{align*}
a_m(x)=\frac{Q_m(x)-x^{m+1}Q_m(1/x)}{1-x},~b_m(x)=\frac{Q_m(x)-x^m Q_m(1/x)}{x-1}.
\end{align*}

\begin{lemma}\label{lemma:aabb}
For any $m\geqslant 0$, we have
\begin{align*}
(1-x) a_{m+1}(x)
&=-2(1+x)(mx-2m+x-1) a_m(x)-2x(1+x)^2 \frac{\mathrm{d}}{\mathrm{d}x}a_m(x)\\
& \quad +x(4mx-x-3) b_m(x)-4x^2(1+x) \frac{\mathrm{d}}{\mathrm{d}x}b_m(x),\\
(x-1) b_{m+1}(x)
&=(4mx+x-1) a_m(x)-4x(1+x) \frac{\mathrm{d}}{\mathrm{d}x}a_m(x) \\
& \quad  + (6mx-4m+x-3)(x+1) b_m(x)- 2x(1+x)^2 \frac{\mathrm{d}}{\mathrm{d}x}b_m(x).
\end{align*}
\end{lemma}
\begin{proof}
Note that
$x^m Q_m(1/x)=x^m a_m(1/x)+x^{m-1} b_m(1/x)=a_m(x)+b_m(x)$.
Setting $\widetilde{Q}_m(x)=x^m Q_m(1/x)$,
it is easy to verify that
\begin{equation}\label{Qmx-recu02}
\widetilde{Q}_{m+1}(x)=(2mx+4m+2x+3)\widetilde{Q}_{m}(x)+2x(1+x)\frac{\mathrm{d}}{\mathrm{d}x}\widetilde{Q}_{m}(x).
\end{equation}
Using~\eqref{Qmx-recu} and~\eqref{Qmx-recu02}, $Q_m(x)=a_m(x)+xb_m(x)$ and $\widetilde{Q}_m(x)=a_m(x)+b_m(x)$ can be rewritten as
\begin{align*}
a_{m+1}(x) + x b_{m+1}(x)
&=(2m+1)(2+3x) a_m(x)
+(6mx + 4m + x)x  b_m(x)\\
& \quad -2x (1+x)\left [\frac{\mathrm{d}}{\mathrm{d}x}a_m(x) + x\frac{\mathrm{d}}{\mathrm{d}x} b_m(x)\right],
\end{align*}
\begin{align*}
a_{m+1}(x) + b_{m+1}(x)
&=( 2mx + 4m + 2x + 3 ) a_m(x)
+( 2mx + 4m + 2x + 3  )  b_m(x)\\
& \quad +2x(1+x)\left [\frac{\mathrm{d}}{\mathrm{d}x}a_m(x) + \frac{\mathrm{d}}{\mathrm{d}x} b_m(x)\right].
\end{align*}
In view of $$(1-x) a_{m+1}(x)=a_{m+1}(x) + x b_{m+1}(x)-x\left(a_{m+1}(x) + b_{m+1}(x)\right),$$
$$(x-1) b_{m+1}(x)=a_{m+1}(x) + x b_{m+1}(x)-\left(a_{m+1}(x) + b_{m+1}(x)\right),$$
it is routine to deduce the desired results.
\end{proof}

Recall that
\begin{align*}
a_m(x) &=\sum_{k=0}^{\lrf{m/2}} \alpha_{m,k} x^k (1+x)^{m-2k},~
b_m(x) =\sum_{k=0}^{\lrf{(m-1)/2}} \beta_{m,k}  x^k (1+x)^{m-1-2k}.
\end{align*}
Now we present the recurrence system of the $\gamma$-coefficients of $a_m(x)$ and $b_m(x)$.
\begin{lemma}\label{lemma:alphabeta}
If $m=2t$ even, then
\begin{align*}
\alpha_{m+1,k} &=2 (2m+1-k)\alpha_{m,k}-(4k-1)\beta_{m,k-1}\quad\text{for $1\leqslant k \leqslant t$};\\
\beta_{m+1,k} &=(1  +4k) \alpha_{m,k} + (4m  +2k+3)\beta_{m,k}\quad\text{for $0\leqslant k \leqslant t-1$};\\
\alpha_{m+1,0}&=2(2m+1)\alpha_{m,0 },\beta_{m+1,t}=(1  +4t) \alpha_{m,t}.
\end{align*}

If $m=2t+1$ odd, then
\begin{align*}
\alpha_{m+1,k} &=2 (  2m+1   -k    ) \alpha_{m,k}   - (  4k-1   )  \beta_{m,k-1}\quad\text{for $1\leqslant k \leqslant t$};\\
\alpha_{m+1,0} &=2 (2m+1)   \alpha_{m,0 },~\alpha_{m+1,t+1}=- (3+4t)  \beta_{m,t};\\
\beta_{m+1,k} &=(1+4k) \alpha_{m,k} + (4m  +2k+3)\beta_{m,k}\quad\text{for $0\leqslant k \leqslant t$}.
\end{align*}
\end{lemma}
\begin{proof}
By Lemma~\ref{lemma:aabb}, after simplifying, we get
\begin{align*}
&(1-x) \sum_{k=0}^{\lrf{(m+1)/2}} \alpha_{m+1,k} x^k (1+x)^{m+1-2k}\\
&=2(1-x)\sum_{k=0}^{m/2}(2m+1-k)\alpha_{m,k} x^k (1+x)^{m+1-2k}\\
&+(x-1)\sum_{k=0}^{(m-1)/2}(4k+3)\beta_{m,k}x^{k+1} (1+x)^{m-1-2k},
\end{align*}
which yields
\begin{align*}
&\sum_{k=0}^{\lrf{(m+1)/2}} \alpha_{m+1,k} x^k (1+x)^{m+1-2k}\\
&=2\sum_{k=0}^{\lrf{m/2}}(2m+1-k)\alpha_{m,k} x^k (1+x)^{m+1-2k}-\sum_{k=1}^{\lrf{(m+1)/2}}(4k-1)\beta_{m,k-1}x^{k} (1+x)^{m+1-2k}.
\end{align*}
Then we get $\alpha_{m+1,k}=2(2m+1-k)\alpha_{m,k}-(4k-1)\beta_{m,k-1}$,
where $\beta_{m,-1}=0$.

Similarly, observe that
\begin{align*}
&(x-1)
\sum_{k=0}^{\lrf{m/2}}\beta_{m+1,k}  x^k (1+x)^{m-2k}\\
&=(4mx+x-1) \sum_{k=0}^{\lrf{m/2}} \alpha_{m,k} x^k (1+x)^{m-2k}+ (6mx-4m+x-3) \sum_{k=0}^{\lrf{(m-1)/2}} \beta_{m,k}  x^k (1+x)^{m-2k}\\
& -4\sum_{k=0}^{\lrf{m/2}} \left(k+mx-kx\right) \alpha_{m,k} x^{k} (1+x)^{m-2k}-2\sum_{k=0}^{\lrf{(m-1)/2}} \left(k-kx+mx-x\right) \beta_{m,k}  x^{k} (1+x)^{m-2k}.
\end{align*}
When $m=2t+1$ odd, we obtain
\begin{align*}
&(x-1)
\sum_{k=0}^{t} \beta_{m+1,k}  x^k (1+x)^{m-2k}\\
&=(4m+1) \sum_{k=0}^{t} \alpha_{m,k} x^{k+1} (1+x)^{m-2k}+ (6m+1)\sum_{k=0}^{t} \beta_{m,k}  x^{k+1} (1+x)^{m-2k}\\
& \quad -4\sum_{k=0}^{t} (m-k) \alpha_{m,k} x^{k+1} (1+x)^{m-2k}-2\sum_{k=0}^{t} ( -k+m-1) \beta_{m,k}  x^{k+1} (1+x)^{m-2k}\\
&-\sum_{k=0}^{t} \alpha_{m,k} x^k (1+x)^{m-2k}-(4m+3)\sum_{k=0}^{t} \beta_{m,k}  x^k (1+x)^{m-2k}\\
& -4\sum_{k=0}^{t} k \alpha_{m,k} x^{k} (1+x)^{m-2k} -2\sum_{k=0}^{t}  k \beta_{m,k}  x^{k} (1+x)^{m-2k}.
\end{align*}
For each $0\leqslant k \leqslant t$, by comparing the coefficients of $x^k (1+x)^{m-2k}$, we find that
$$\beta_{m+1,k}=(1+4k)\alpha_{m,k}+(4m+2k+3)\beta_{m,k}.$$

When $m=2t$ even, we find
\begin{align*}
&(x-1)
\sum_{k=0}^{t} \beta_{m+1,k}  x^k (1+x)^{m-2k}\\
&=(4m+1) \sum_{k=0}^{t} \alpha_{m,k} x^{k+1} (1+x)^{m-2k}+ (6m+1)\sum_{k=0}^{t-1} \beta_{m,k}  x^{k+1} (1+x)^{m-2k}\\
& -4\sum_{k=0}^{t} (m-k) \alpha_{m,k} x^{k+1} (1+x)^{m-2k}-2\sum_{k=0}^{t-1} ( -k+m-1) \beta_{m,k}  x^{k+1} (1+x)^{m-2k}\\
&-\sum_{k=0}^{t} \alpha_{m,k} x^k (1+x)^{m-2k}-(4m+3)\sum_{k=0}^{t-1} \beta_{m,k}  x^k (1+x)^{m-2k}\\
&-4\sum_{k=0}^{t} k \alpha_{m,k} x^{k} (1+x)^{m-2k} -2\sum_{k=0}^{t-1}  k \beta_{m,k}  x^{k} (1+x)^{m-2k}.
\end{align*}
For each $0\leqslant k \leqslant t-1$, by comparing the coefficients of $x^k (1+x)^{m-2k}$, we get
$\beta_{m+1,k}=(1+4k)\alpha_{m,k}+(4m+2k+3)\beta_{m,k}$.
For $k=t$, we have $\beta_{m+1,t}=(1+4t)\alpha_{m,t}$.
\end{proof}

\noindent{\bf A proof of Theorem~\ref{thm02}:}
\begin{proof}
Note that $\alpha_{1,0}=2$ and $\beta_{1,0}=1$,
$\alpha_{2,0}=12,\alpha_{2,1}=-3$ and $\beta_{2,0}=9$.
By induction on $m$, we assume that both $(-1)^k\alpha_{m,k} $ and $(-1)^k\beta_{m,k}$ are positive integers.
By Lemma~\ref{lemma:alphabeta}, we obtain
\begin{align*}
\sgn(\alpha_{m+1,k}) &=\sgn\left(\alpha_{m,k})=-\sgn(\beta_{m,k-1}\right),~
\sgn(\beta_{m+1,k})=\sgn(\alpha_{m,k})=\sgn(\beta_{m,k}).
\end{align*}
And so both $(-1)^k\alpha_{m+1,k} $ and $(-1)^k\beta_{m+1,k}$ are positive integers. This completes the proof.
\end{proof}
\section*{Acknowledgements}
The second author was supported by the National Natural Science Foundation of China (Grant number 12071063).

\end{document}